\documentclass[a4,12pt]{article}
\usepackage{latexsym,amsfonts,amsmath,amsthm,amssymb, epsfig}
\usepackage{color}
\usepackage{geometry}
\theoremstyle{theorem}
\newtheorem{theorem}{Theorem}

\theoremstyle{definition}

\newcommand{\RR}{\mathbb{R}}

\newcommand{\NN}{\mathbb{N}}

\newcommand{\EE}{\mathbb{E}}

\newcommand{\ee}{{\rm e}}

\allowdisplaybreaks

\title{Bernstein approximation and beyond: proofs by means of elementary probability theory}  
\author{Tiangang Cui and Friedrich Pillichshammer}

\begin{document}

\maketitle

\begin{abstract}
Bernstein polynomials provide a constructive proof for the Weierstrass approximation theorem, which states that every continuous function on a closed bounded interval can be uniformly approximated by polynomials with arbitrary accuracy. Interestingly the proof of this result can be done using elementary probability theory. This way one can even get error bounds for Lipschitz functions. In this note, we present these techniques and show how the method can be extended naturally to other interesting situations. As examples, we obtain in an elementary way results for the Sz\'{a}sz-Mirakjan operator and the Baskakov operator.      
\end{abstract}

\section{Bernstein approximation}

For $n \in \NN$ and $k \in \{0,1,\ldots,n\}$ the $k$-th Bernstein basis polynomial of degree $n$ is defined as
\begin{eqnarray*}
\beta_{n}(k,x) := \binom{n}{k} x^k (1-x)^{n-k} \quad \mbox{for $x \in [0,1]$}.
\end{eqnarray*}
For a continuous function $f:[0,1] \to \mathbb{R}$ and $n \in \NN$, Bernstein~\cite{b12} constructed an approximation scheme in the form of 
\begin{equation*}
B_n(f;x) := \sum_{k=0}^n  f\left(\frac{k}{n}\right) \beta_{n}(k,x), 
\end{equation*}
to prove the Weierstrass approximation theorem. His proof is based on methods from elementary probability theory (see also \cite[Proposition~5.2]{EL}). Kac~\cite{K38} gave a formula for the approximation error for Lipschitz continuous functions. These results are also discussed by Math\'{e} in \cite{M99}, which is worth reading also because of the interesting historical comments. 

\begin{theorem}\label{th0}
Assume that a function $f:[0,1]\rightarrow \mathbb{R}$ is $\alpha$-H\"older continuous, i.e., $f \in C^{0,\alpha}([0,1])$ with $\alpha \in (0,1]$, such that 
\begin{equation}\label{Hcond1}
|f(x)-f(y)| \le L |x-y|^\alpha\quad \mbox{for all $x,y \in [0,1]$,}
\end{equation}
for some real constant $L>0$. Then for all $n \in \mathbb{N}$ and all $x \in [0,1]$ we have
\begin{equation*}
|f(x)- B_n(f;x)| \le L \left( \frac{x(1-x)}{n}\right)^{\alpha/ 2}.
\end{equation*}
\end{theorem}

We briefly revisit the proof of this result according to Math\'{e}~\cite{M99} which uses elementary arguments from probability theory. Thereby we make some minor reformulations that allow us to get a clearer picture of the overall situation, which in turn makes it clear how the method can be expanded.

\begin{proof}[Proof of Theorem~\ref{th0}]
Let  $x \in [0,1]$ be the success probability of a Bernoulli experiment and $n \in \mathbb{N}$ be the total number of experiments. The number of successful Bernoulli experiments, which is a discrete random variable denoted by $K$, follows the binomial distribution with the probability mass function
$$\beta_{n}(k,x)= \binom{n}{k} x^k (1-x)^{n-k}\quad \mbox{for $k \in \{0,1,\ldots,n\}$}.$$ 
The probability of $K=k$ is given by the $k$-th Bernstein basis polynomial. We have the expectation $\EE[K]=n x$ and the variance ${\rm Var}[K]= n x (1-x)$. 

For a function $f:[0,1]\rightarrow \mathbb{R}$ we have the identity
\[
\mathbb{E}\left[f\left(\frac{K}{n}\right)\right] = \sum_{k = 0}^n f\left(\frac{k}{n}\right) \binom{n}{k} x^k (1-x)^{n-k} = B_n(f;x), 
\]
and thus
$$
|f(x)-B_n(f;x)| = \left|\EE\left[f(x)- f\left(\frac{K}{n}\right)\right]\right| \le  \EE\left[\left|f(x)- f\left(\frac{K}{n}\right)\right|\right].
$$
Assuming that $f\in C^{0,\alpha}([0,1])$ and using the above inequality, we then have
\begin{equation}\label{th0eq1}
|f(x)-B_n(f;x)| \leq L \, \EE\left[\left|x- \frac{K}{n}\right|^{\alpha}\right] = \frac{L}{n^{\alpha}} \EE\left[\left|n x- K\right|^{\alpha}\right].
\end{equation}
Applying H\"older's inequality with parameters $2/\alpha$ and $2/(2-\alpha)$, we obtain
\begin{equation}\label{th0eq2}
\EE\left[\left|n x- K\right|^{\alpha}\right] \leq \left(\EE\left[\left|n x- K\right|^2\right]\right)^{\alpha/2}  =  {\rm Var}[K]^{\alpha/2} = \left(n x (1-x)\right)^{\alpha/2} .
\end{equation}
Substituting \eqref{th0eq2} into \eqref{th0eq1}, the result follows. 
\end{proof}

The key insight we gain here is that one can exploit the equivalence between discrete probability distributions and certain basis functions to construct approximation schemes. This note aims to demonstrate how this method can be easily adapted to other discrete probability distributions and thereby to other approximation schemes rather than that based on Bernstein basis polynomials. This approach leads in an easy way to results for the Sz\'{a}sz-Mirakjan operator and the Baskakov operator.

\section{Other approximations schemes}
We aim to approximate a continuous function $f:[0,\infty)\rightarrow \mathbb{R}$ using a discrete random variable $K$ taking values in $\NN_0$. Let the probability mass function of $K$ be $p_{n}(k,x)$, where $n \in \NN$ and $x \in [0,\infty)$ are parameters. Assuming $\EE[K]=n x$ and ${\rm Var}[K]< \infty$, we define the approximation scheme
$$S_n(f;x):=\sum_{k = 0}^{\infty} f\left(\frac{k}{n}\right) p_{n}(k,x).$$ 
The following theorem estimates the approximation error of $S_n(f;x)$.
\begin{theorem}\label{th1}
Assume that a function $f:[0,\infty)\rightarrow \mathbb{R}$ satisfies a general H\"older-type condition, 
\begin{equation}\label{Hcond4}
|f(x)-f(y)| \le L \, \frac{|x-y|^{\alpha}}{(\gamma+x+y)^{\beta}} \quad \mbox{for all $x,y \in [0,\infty)$,}
\end{equation}
for real constants $L>0$, $\alpha \in (0,1]$ and $\beta,\gamma \in [0, \infty)$, then for all $n \in \mathbb{N}$ and all $x \in (0,\infty)$ we have 
\begin{equation*}
|f(x)-S_n(f;x)| \le  \frac{L}{n^{\alpha}} \frac{{\rm Var}[K]^{\alpha/2}}{(\gamma+x)^\beta},
\end{equation*}
where $K$ is a $\NN_0$-valued random variable with probability mass function $p_{n}(k,x)$.
\end{theorem}

For a function $f$ satisfying the general H\"older-type condition \eqref{Hcond4}, we denote it by $f \in C^{0, \alpha, \beta, \gamma}([0, \infty))$. The constant $\gamma$ is used to avoid singularities at the origin in error estimates. With increasing $\beta$, we impose a stronger decaying rate on $f$ toward its tail. With $\beta = 0$, the general H\"older-type condition reduces to the plain H\"older condition
\begin{equation}\label{Hcond5}
|f(x)-f(y)| \le L |x-y|^\alpha\quad \mbox{for all $x,y \in [0,\infty)$,}
\end{equation}
for $L>0$ and $\alpha \in (0,1]$. In this case, the error estimate in Theorem \ref{th1} becomes 
$$
|f(x)-S_n(f;x)| \le  \frac{L}{n^{\alpha}} {\rm Var}[K]^{\alpha/2}.
$$ 

\begin{proof}[Proof of Theorem~\ref{th1}]
Let $f : [0,\infty) \rightarrow \mathbb{R}$ be a continuous function and $K \in \NN_0$ be a discrete random variable with probability mass function $p_{n}(k,x)$, where $n \in \NN$ and $x \in (0,\infty)$. We have
\[
\mathbb{E}\left[f\left(\frac{K}{n}\right)\right] = \sum_{k = 0}^{\infty} f\left(\frac{k}{n}\right) p_{n}(k,x)=S_n(f;x),
\]
and thus
\[
|f(x)-S_n(f;x)| = \left|\EE\left[f(x)- f\left(\frac{K}{n}\right)\right]\right| \le  \EE\left[\left|f(x)- f\left(\frac{K}{n}\right)\right|\right].
\]
Assuming $f \in C^{0, \alpha, \beta, \gamma}([0, \infty))$ and using the above inequality, we have 
\begin{equation}\label{th1eq1}
|f(x)-S_n(f;x)| \leq L \, \EE\left[\frac{\left|x- \frac{K}{n}\right|^{\alpha}}{\left(\gamma + x + \frac{K}{n}\right)^{\beta}}\right] \leq \frac{L}{n^{\alpha}} \frac{\EE\left[\left|nx- K\right|^{\alpha}\right]}{\left(\gamma+x\right)^{\beta}}.
\end{equation}
Following the argument in \eqref{th0eq2}, we have $\EE\left[\left|n x- K\right|^{\alpha}\right] \leq {\rm Var}[K]^{\alpha/2}$, and therefore the result follows. 
\end{proof}

As an application we obtain results for the Sz\'{a}sz-Mirakjan operator and the Baskakov operator.

\paragraph{Sz\'{a}sz-Mirakjan operator.}  The Sz\'{a}sz-Mirakjan operator \cite{sza} is based on the Poisson distribution $P_{\lambda}$ with parameter $\lambda>0$, which has the probability mass function
\[
p(k) = \ee^{-\lambda} \frac{\lambda^k}{k!}\quad \mbox{for $k \in \NN_0$.}
\]
If $K$ is distributed according to $P_\lambda$, which is written as $K \sim P_\lambda$, then it is well known that $\mathbb{E}[K] = {\rm Var}[K]=\lambda$. Now, for $n \in \NN$ and $x \in [0,\infty)$ we define the parameter $\lambda=n x$ and let $K\sim P_{nx}$. This defines the basis functions 
$$
p_{n}(k,x) = \ee^{-n x} \frac{(n x)^k}{k!} \quad \mbox{for $k \in \NN_0$.}
$$
Then, for $f \in C^{0,\alpha, \beta, \gamma}([0,\infty))$ and $n \in \NN$ we have the approximation scheme 
$$S_n(f;x):=\sum_{k = 0}^{\infty} f\left(\frac{k}{n}\right) \ee^{-n x} \frac{(n x)^k}{k!},$$ 
which is known as Sz\'{a}sz-Mirakjan operator; see, e.g., \cite{sza,VT} for a detailed study using analytic methods. From Theorem~\ref{th1} we obtain its error estimate
$$
|f(x)- S_n(f;x)| \le \frac{L}{n^{\alpha/2}} \frac{x^{\alpha/2}}{(\gamma+x)^\beta},
$$
and some specific cases.
\begin{enumerate}
\item If $f:[0,\infty) \to \mathbb{R}$ satisfies the plain H\"older condition \eqref{Hcond5}, then for all $n \in \mathbb{N}$ and all $x \in (0,\infty)$ we have
\begin{equation*}
|f(x)- S_n(f;x)| \le \frac{L}{n^{\alpha/2}}\, x^{\alpha/2}  .
\end{equation*}
\item If $f$ satisfies the general H\"older-type condition \eqref{Hcond4} with $\gamma = 0$ and $\beta = \alpha/2$, then for all $n \in \mathbb{N}$ and uniformly for all $x \in (0,\infty)$ we have the bound 
\begin{equation}\label{SMOp_unif}
|f(x)-S_n(f;x)| \le \frac{L}{n^{\alpha/2}} .
\end{equation}
\item If $f$ satisfies the general H\"older-type condition \eqref{Hcond4} with $\beta > \alpha/2$, for all $n \in \mathbb{N}$ and for all $x \in (0,\infty)$ we have the bound
\begin{equation*}%\label{SMOp_unif2}
|f(x)-S_n(f;x)| \le \frac{L}{n^{\alpha/2}} (\gamma+x)^{-\beta + \alpha/2},
\end{equation*} 
which decays toward the tail as $x$ increases. This also yields a similar uniform upper bound as that in \eqref{SMOp_unif}.
\end{enumerate}

\paragraph{Baskakov operator.}

The Baskakov operator \cite{b57} is based on the Pascal distribution $PC_{n,x}$ (also known as negative binomial distribution) with parameters $n \in \NN$ and $x \in [0,\infty)$, which has the probability mass function 
$$p_{n}(k,x)=\binom{n+k-1}{k} \frac{x^k}{(1+x)^{n+k}} \quad \mbox{for $k \in \NN_0$}.$$ 
If $K \sim PC_{n,x}$, then we have $\EE[K] = n x$ and ${\rm Var}[K]= nx(1+x)$. Then, for $f \in C^{0,\alpha,\beta,\gamma}([0,\infty))$ and $n \in \NN$ we have  the approximation scheme 
$$S_n(f;x)=\sum_{k = 0}^{\infty} f\left(\frac{k}{n}\right) \binom{n+k-1}{k} \frac{x^k}{(1+x)^{n+k}}.$$ 
This operator is known as Baskakov operator in literature; see, e.g., \cite{b57,VT}. From Theorem~\ref{th1} we obtain its error estimate
$$
|f(x)- S_n(f;x)| \le \frac{L}{n^{\alpha/2}} \frac{(x(x+1))^{\alpha/2}}{(\gamma+x)^\beta},
$$
and the following specific cases. 
\begin{enumerate}
\item If $f:[0,\infty) \to \mathbb{R}$ satisfies the plain H\"older condition \eqref{Hcond5}, then for all $n \in \mathbb{N}$ and all $x \in (0,\infty)$ we have
\begin{equation*}
|f(x)- S_n(f;x)| \le L \left( \frac{x(x+1)}{n}\right)^{\alpha/ 2}.
\end{equation*}
\item If $f$ satisfies the general H\"older-type condition \eqref{Hcond4} with $\gamma=0$ and $\beta = \alpha/2$, then for all $n \in \mathbb{N}$ and all $x \in (0,\infty)$ we have 
\begin{equation*}
|f(x)-S_n(f;x)| \le  L \left( \frac{x+1}{n}\right)^{\alpha/2}.
\end{equation*}
\item With a stronger decay rate $\beta\geq \alpha$, we can also extend the convergence rate of the Baskakov operator uniformly for all $x\in [0, \infty)$. For example, with $\gamma = 1$ and $\beta = \alpha$, for all $n \in \mathbb{N}$ and uniformly for all $x \in (0,\infty)$ we have
\begin{equation*}
|f(x)-S_n(f;x)| \le  \frac{L}{n^{\alpha/2}}.
\end{equation*}
\end{enumerate}

\paragraph{Final remark and examples.}
For both Sz\'{a}sz-Mirakjan and Baskakov operators, it is possible to bound the error $|f(x)-S_{n}(f;x)|$ uniformly for all $x \in [0, \infty)$ for functions with a sufficiently large decay rate $\beta$. However, their convergence still follows a rate of $n^{-\alpha/2}$, and thus we can recover a convergence rate of $n^{-1/2}$ with $\alpha = 1$ as the best-case scenario.

To numerically implement these operators, we also require a truncation in $k$ in the approximation. For $m \in \NN$, we have the truncated approximation scheme
\begin{equation}\label{trunc_approx}
\widetilde{S}_{n,m}(f;x):=\sum_{k=0}^{m} f\left(\frac{k}{n}\right) p_{n}(k,x).
\end{equation}
The error of the approximation \eqref{trunc_approx} has an apparent upper bound
$$
|f(x)-\widetilde{S}_{n,m}(f;x)| \le  |f(x)-S_{n,m}(f;x)| + |S_n(f;x)-\widetilde{S}_{n,m}(f;x)|,
$$
in which the first term of the bound is given by Theorem \ref{th1} and the second term of the bound (the {\it truncation error}) satisfies 
\begin{eqnarray}
|S_n(f;x)-\widetilde{S}_{n,m}(f;x)| & = & \left| \sum_{k > m} f\left(\frac{k}{n}\right) p_{n}(k,x) \right| \nonumber \\
& \leq & \sup_{x\in(m/n, \infty)} \big| f(x) \big| \, \sum_{k > m} p_{n}(k,x) \nonumber \\
& = & \sup_{x\in(m/n, \infty)} \big| f(x) \big| \, \mathbb{P}[K>m]. \label{trunc_error}
\end{eqnarray}

We first consider a function $f \in C^{0,\alpha,\beta,\gamma}([0,\infty))$ with a sufficiently large $\beta$ that leads to a uniform error bound $|f(x)-S_{n,m}(f;x)| \leq L \, n^{-\alpha/2}$ in the untruncated approximation. For $x \in (m/n,\infty)$, we further assume that $f$ satisfies a tail condition $|f(x)| = O(g(x))$ for $x \rightarrow \infty$ for some strictly decreasing function $g: (0,\infty) \rightarrow \mathbb{R}^{+}$. Applying the bound in \eqref{trunc_error} and $\mathbb{P}[K>m] \leq 1$, the truncation error satisfies 
$$
|S_n(f;x)-\widetilde{S}_{n,m}(f;x)| = O(g(m/n)).
$$
Then, using $m = \lceil n g^{-1}(n^{-\alpha/2}) \rceil$, the error of the truncated approximation scheme satisfies $|f(x)-\widetilde{S}_{n,m}(f;x)| = O(n^{-\alpha/2}).$ The following are some examples. 
\begin{enumerate}
\item The function $f(x) = (1+x^2)^{-1}$ satisfies the general H\"{o}der condition \eqref{Hcond4} with $\alpha = 1$ and $\beta = 1$. 
We also have $f(x) = O(x^{-2})$ for $x \rightarrow \infty$. Thus, we can choose $m = \lceil n^{5/4} \rceil$ so that the truncated approximation satisfies $|f(x)-\widetilde{S}_{n,m}(f;x)| = O(n^{-1/2}).$
\item The function $f(x) = \exp(-x)$ satisfies the general H\"{o}der condition \eqref{Hcond4} with $\alpha = 1$ and $\beta = 1$. Thus, we can choose $m = \lceil n \log(n) / 2 \rceil$ so that the truncated approximation satisfies $|f(x)-\widetilde{S}_{n,m}(f;x)| = O(n^{-1/2}).$

\end{enumerate}

If the function $f:[0,\infty) \rightarrow \RR$ does not necessarily satisfy a tail condition to guide the truncation in $k$, but is at least bounded, then we can use an alternative argument from elementary probability theory. Assuming  $\|f\|_{\infty}:=\sup_{x\in [0,\infty)}|f(x)|< \infty$, the bound of the truncation error in \eqref{trunc_error} also leads to
$$
|S_n(f;x)-\widetilde{S}_{n,m}(f;x)| \le \|f\|_{\infty} \ \mathbb{P}[K>m],
$$ 
in which $\mathbb{P}[K>m]$ can be estimated using Chebyshev's inequality. Choosing $m \ge 2 \EE[K]$ we have 
\begin{eqnarray*}
\mathbb{P}[K>m] \le \mathbb{P}\Big[\big|K-\EE[K] \big|\ge \EE[K]\Big] \le \frac{{\rm Var}[K]}{\EE[K]^2}.
\end{eqnarray*}
In the case of the Sz\'{a}sz-Mirakjan operator, we have $\EE[K]={\rm Var}[K]=n x$, and thus choosing $m= 2 \lceil n x \rceil$ we obtain $$|S_n(f;x)-\widetilde{S}_{n,m}(f;x)| \le \|f\|_{\infty} \frac{1}{n x}.$$
In the case of the Baskakov operator, we have $\EE[K]=nx$ and ${\rm Var}[K]=n x (1+x)$. Hence, choosing again $m= 2 \lceil n x \rceil$ we obtain $$|S_n(f;x)-\widetilde{S}_{n,m}(f;x)| \le \|f\|_{\infty} \frac{1}{n} \frac{1+x}{x}.$$

\iffalse
So the truncation error depends mostly on the decay rate of the probability mass function $p_{n,x}$ when $k \rightarrow \infty$.

For example, in the case of the truncated Sz\'{a}sz-Mirakjan operator we have
$$\sum_{k>M} p_{n,x}(k) = \ee^{-n x} \sum_{k>M} \frac{(n x)^k}{k!} = \ee^{- n x} \frac{\ee^{\xi}}{(M+1)!} (n x)^{M+1},$$
where $\xi \in (0,nx)$. For the last identity we used the Lagrange form of the remainder of the exponential series. Using $$\frac{1}{n!} < \frac{1}{\sqrt{2 \pi n}} \left(\frac{\ee}{n}\right)^n \quad \mbox{for $n \in \NN$,}$$ which follows easily from Stirlings formula, we obtain
$$\sum_{k>M} p_{n,x}(k) \le \frac{1}{\sqrt{2 \pi (M+1)}}  \left(\frac{n x \ee}{M+1}\right)^{M+1}.$$ Choosing $M=\lfloor 2 {\rm e} x n\rfloor$ we therefore obtain $$\sum_{k>M} p_{n,x}(k) \le \frac{1}{\sqrt{2 \pi (\lfloor 2 {\rm e} x n\rfloor+1)}} \left(\frac{1}{2}\right)^{2 {\rm e} x n}.$$
\fi

\end{document}